\documentclass[12pt]{article}

\usepackage{amsfonts}
\usepackage{amsmath}
\usepackage{amssymb}
\usepackage{amsthm}

\def\calP{\mathcal{P}}

\def\calD{\mathcal{D}}

\def\psq{\boxminus}

\def\ra{\rightarrow}

\def \R {\mathbb R}

\def\al{\alpha}

\def\be{\beta}

\def\de{\delta}

\def\ga{\gamma}

\def\si{\sigma}
\def\de{{\delta}}

\def\om{\omega}
\def\Om{\Omega}
\def\mabg{\mu^{(\alpha; \beta, \gamma)}}

\def\Ph1{P^{(h_1)}}
\def\Ph2{P^{(h_2)}}

\def\b0{{\bf 0}}

\def\bnu{\bar\nu}

\def\1lbnu{{\bnu}_{\lambda_1}}
\def\2lbnu{{\bnu}_{\lambda_2}}

\newtheorem{thm}{Theorem}[section]
\newtheorem{prop}[thm]{Proposition}
\newtheorem{lem}[thm]{Lemma}
\newtheorem{cor}[thm]{Corollary}

\theoremstyle{plain}

\newtheorem{defn}[thm]{Definition}

\newtheorem{claim}[thm]{Claim}

\begin{document}

\title{BK-type inequalities and generalized random-cluster representations}

\author{J. van den Berg\footnote{CWI and VU University Amsterdam} 
and A. Gandolfi\footnote{Dipartimento di Matematica U. Dini, Univ. di Firenze} \\
{\footnotesize email: J.van.den.Berg@cwi.nl, gandolfi@math.unifi.it}
}

\date{}
\maketitle

\begin{abstract}
%
Recently, van den Berg and Jonasson gave the first substantial extension of the BK inequality
for non-product measures: they proved that, for $k$-out-of-$n$ measures, the probability that
two increasing events occur disjointly is at most the product of the two individual probabilities.

We show several other extensions and modifications of the BK inequality.
In particular, we prove that the antiferromagnetic Ising Curie-Weiss model satisfies the BK
inequality for all increasing events. We prove that this also holds for the Curie-Weiss model with
three-body interactions under the so-called Negative Lattice Condition.
For the ferromagnetic
Ising model we show that the probability that two events occur `cluster-disjointly' is at most the product
of the two individual probabilities, and we give a more abstract form of this result for arbitrary Gibbs measures.

The above cases are derived from a general abstract theorem whose proof is based on an extension of the
Fortuin-Kasteleyn random-cluster representation for all probability distributions and on a `folding procedure' which
generalizes an argument of D. Reimer.

\end{abstract}

\begin{section}{Introduction and statement of results}

\begin{subsection}{Definitions, background and overview} \label{sect-defs}
Before we state and discuss results in the literature that are needed in, or partly motivated, our current work, 
we introduce the main definitions and notation:
Let $S$ be a finite set and let $\Om$ denote the set $S^n$. This set will be our state space.
We will often use the notation $[n]$ for $\{1, \cdots,n\}$, the set of indices.
For $\om \in \Om$ and $K \subset [n]$, we define $\om_K$ as the `tuple' $(\om_i, i \in K)$.
We use the notation $[\om]_K$ for the set of all elements of $\Om$ that `agree with $\om$ on $K$'. More formally,
$$[\om]_K := \{\al \in \Om \, : \, \al_K = \om_K\}.$$

For $A, B \subset \Om$, $A \square B$ is defined as the event that $A$ and $B$ `occur disjointly'.
Formally, the definition is:

\begin{equation} \label{box-def}
A \square B = \{\om \in \Om \, : \, \exists \text{ disjoint } K, L \subset [n] \mbox{ s.t. } 
[\om]_K \subset A \text{ and } [\om]_L \subset B\}.
\end{equation}

For the case where $S$ is an ordered set and $\om$ and $\om' \in \Om$, we write $\om' \geq \om$
if $\om'_i \geq \om_i$ for all $i \in [n]$.
An event $A \subset \Om$ is said to be increasing if $\om' \in A$ whenever $\om \in A$ and $\om' \geq \om$. \\

The following inequality, \eqref{bkr-eq} below,
was conjectured (and proved for the special case where $S = \{0,1\}$ and $A$ and $B$ are increasing events)
in \cite{BK85}. Some other special cases were proved in \cite{BF87} and \cite{T94}. The general case was proved
by D. Reimer (see \cite{R00}).

\begin{thm} \label{bkr}
For all $n$, all product measures $\mu$ on $S^n$,  and all $A, B \subset S^n$,
\begin{equation}\label{bkr-eq}
\mu(A \square B) \leq \mu(A) \mu(B).
\end{equation}
\end{thm}

We also state the following result, Proposition \ref{reimer} below, which was Reimer's key ingredient (intermediate
result) in his
proof of Theorem \ref{bkr}, and which is also crucial in our work. 

To state this result, some more notation is needed:
For $\om = (\om_1, \cdots, \om_n) \in \{0,1\}^n$, we denote by $\bar\om$ the configuration
obtained from $\om$ by replacing $1$'s by $0$'s and vice versa:
$$\bar\om = (1-\om_1, \cdots, 1-\om_n).$$
Further, for $A \subset \Om$, we define
$\bar A \, = \, \{\bar\om \, : \, \om\in A\}$.
Finally, if $V$ is a finite set, $|V|$ denotes the number of elements of $V$.

\begin{prop}\label{reimer} {\em [Reimer \cite{R00}]} \\
For all $n$ and all $A, B \subset \{0,1\}^n$,
\begin{equation}\label{reimer-eq}
|A \square B| \leq |A \cap \bar B|.
\end{equation}
\end{prop}

\medskip
It is easy to see that if $\mu$ is a
non-product measure on $S^n$, it cannot satisfy \eqref{bkr-eq} for all events.
However, it seemed intuitively obvious that many measures on $\{0,1\}^n$ do satisfy this inequality for all
{\em increasing} events. Such measures are sometimes called BK measures. The first natural, non-trivial, non-product measure
which was proved to be BK is the so-called $k$-out-of-$n$ measure:

Let $k \leq n$ and let $\Om_{k,n}$ be the set of all
$\om \in \{0,1\}^n$ with exactly $k$ $1's$.
Let $P_{k,n}$ be the distribution on $\{0,1\}^n$ that assigns
equal probability to all $\om \in \Om_{k,n}$ and probability $0$ to all other elements of $\{0,1\}^n$.

\begin{thm} \label{bkj} {\em [van den Berg and Jonasson \cite{BJ11}]}
For all $n$, all $k \leq n$, and all increasing $A, B \subset \{0,1\}^n$,
\begin{equation}\label{bkj-eq}
P_{k,n}(A \square B) \leq P_{k,n}(A) P_{k,n}(B).
\end{equation}
\end{thm}

\noindent
{\bf Remark:}
{\em This result, which was conjectured in \cite{G94}, extends, 
as pointed out in \cite{BJ11}, to certain weighted versions of 
$P_{k,n}$ and to products of such measures.}

Theorem \ref{bkj} is one direction in which Theorem \ref{bkr} can be extended or generalized. In our
current work we prove some other natural cases, including the antiferromagnetic Curie-Weiss model (see Theorem \ref{afcw-thm}
and Theorem \ref{afcw-cubic-thm}), in this direction. However, we
also generalize Theorem \ref{bkr} in a very different sense, namely by modifying the
disjoint-occurrence operation. In particular we will show that the ferromagnetic Ising model satisfies
\eqref{bkr-eq} for the modification where the usual `disjoint-occurrence' notion is replaced by the stronger
notion of disjoint-spin-cluster occurrence (see Theorem \ref{fim-thm}). (A form of this result is also proved
for arbitrary Gibbs measures, see Theorem \ref{gibbs-thm}). As an example we derive an upper bound
for the probability of a certain four-arm event in terms of the one-arm probabilities (see Corollary \ref{ising-cor}).

All these results are stated in Section \ref{sect-newext}. Next, in Section \ref{sect-general} we state and prove
a very general result,
Theorem \ref{bk-gen}. This theorem involves the notion of `foldings' of a measure, which already plays (but only
within the class of product measures and therefore less explicitly) an important role in Reimer's work \cite{R00}.
This notion is defined in Section \ref{sect-foldings}. Theorem \ref{bk-gen} also involves a highly generalized form of
the Fortuin-Kasteleyn random-cluster
representation for all probability distributions, presented in Section \ref{sect-gen-rcr}. In very simplified and informal
terms, Theorem \ref{bk-gen} states that an inequality similar to \eqref{bkr-eq} holds whenever the events $A$
and $B$ and the probability distribution $\mu$ are such that if (a certain version of) $A \square B$ holds,
$A$ and $B$ can be `witnessed' by sets of indices that are not connected to each other in the random-cluster
configurations for the foldings of $\mu$. In Section \ref{gibbs-proofs} we investigate random-cluster
representations for Gibbs measures and show how our general result implies the above mentioned Theorem \ref{gibbs-thm}
and Theorem \ref{fim-thm}. Finally, in Section \ref{afcw-proof} we derive our theorems for the Curie-Weiss model
from the general result. 

We finish the current section by remarking that a very different kind of extension (namely, a `dual form') of Theorem \ref{bkr}, within the
class of product-measures, was obtained by Kahn, Saks and Smyth \cite{KSS11}.

\end{subsection}

\begin{subsection}{New extensions of Theorem \ref{bkr}} \label{sect-newext}
\begin{subsubsection}{Antiferromagnetic Ising Curie-Weiss model} \label{sect-acw}
In this model each pair of vertices has the same, antiferromagnetic, interaction. Moreover, each vertex
`feels' an external field (which may be different from that at the other vertices). 

More precisely, the Ising Curie-Weiss measure with vertices $\{1, \cdots, n\}$, interaction
parameter $J$ and external fields $h_1, \cdots, h_n$, is the distribution $\mu$ on
$\{-1, +1\}^n$ given by

\begin{equation}\label{afcw-def}
\mu(\om) = \frac{\exp\left(\sum_{i, j} J \om_i \om_j \, + \, \sum_i h_i \om_i\right)}{Z}, \,\,\, \om \in \{-1, +1\}^n,
\end{equation}
where (here and in similar expressions later)
$Z$ is a normalizing constant, the first sum is over all pairs $(i,j)$ with $1 \leq i < j \leq n$ and
the second sum is over all $i = 1, \cdots, n$.
If $J < 0$, the measure is called antiferromagnetic.
One of our main new results is that in that case it is a BK measure:

\begin{thm} \label{afcw-thm}
The Ising Curie-Weiss measure \eqref{afcw-def} with $J \leq 0$ satisfies
\begin{equation} \label{bk-afcw}
\mu(A \square B) \leq \mu(A) \, \mu(B),
\end{equation}
for all increasing $A, B \subset \{-1, +1\}^n$.
\end{thm}

\noindent
In Section \ref{afcw-proof} we prove this theorem from the more general Theorem \ref{bk-gen}.

\smallskip\noindent
{\bf Remark:}
{\it 
If $n$ is even and all the $h_i's$ are $0$, letting $J \rightarrow  -\infty$ in \eqref{afcw-def} yields
the $n/2$-out-of-$n$ distribution (with $-1$ playing the role of $0$). More generally, taking all the $h_i$'s equal to
a common value $h$, and then letting $J \rightarrow  -\infty$ and simultaneously (in a suitable way, depending
on $k$) $h \rightarrow \infty$ (or $ - \infty$), yields the $k$-out-of-$n$ distribution. In this sense Theorem
\ref{bkj} can be seen as a special case of Theorem \ref{afcw-thm} above. In fact, the proof of
Theorem \ref{afcw-thm} is somewhat similar in spirit to that of Theorem \ref{bkj}.
Roughly speaking, it boils down to
showing that an antiferromagnetic Curie-Weiss measure without external fields on
$\{-1,+1\}^n$ can be written as a convex combination of products of `independent fair coin-flips'
and $1$-out-of-$2$ measures.
However, to show that such convex combination exists is much more involved than the
analogous work in \cite{BJ11} for the $k$-out-of-$n$ model (and is not intuitively obvious at all).
}

\begin{subsubsection}{Curie-Weiss model with three-body interactions} \label{cw-3b}
It is well-known and easy to see that the Ising Curie-Weiss measure \eqref{afcw-def} satisfies
the negative lattice condition
\begin{eqnarray} \label{nlc}
\mu(\omega \vee \omega') \mu(\omega \wedge \omega')
\leq \mu(\omega)\mu(\omega') \quad \quad \text{ for all } \omega, \omega' \in \Omega
\end{eqnarray}
if $J \leq 0$, and that it satisfies
the positive lattice condition (i.e. \eqref{nlc}
with reverse inequality) if $J \geq 0$.
Thus Theorem \eqref{afcw-thm} says that for the Ising Curie-Weiss model the
negative lattice condition implies the BK property, while
it was only known to imply negative association (see \cite{P00}), a property which is weaker than
BK (see \cite{M09}).

One could wonder if this is the case also for the Curie-Weiss model with
multibody interaction. 
This question is in some sense opposite to those in \cite{LST07} where they deal with 
infinite extendibility (IE) to an exchangeable distribution, a property which implies the positive 
lattice condition.


We investigate here only the first step in this direction, namely the
addition of a three-body interaction to \eqref{afcw-def}:
combining Lemmas \eqref{lem-lin-sol} and \eqref{lem-cw-sol} and the remark that the cubic part of the
interaction disappears in the foldings, we show that even for this model the negative
lattice condition implies BK. With $S= \{-1,1\}$ the Curie-Weiss model with
three-body interactions $\mu$ is the distribution on
$\Omega=S^n$ given by
\begin{equation}\label{afcw-cubic-def}
\mu(\om) = \frac{\exp\left(\,h \sum_i  \om_i + J_2\sum_{i, j}  \om_i \om_j \, + J_3
\sum_{i, j,k}  \om_i \om_j   \om_k\ \right)}{Z}, \,\,\, \om \in \{-1, 1\}^n,
\end{equation}
where in the last two  sums we take $i< j$ and $i<j<k$,
respectively.
\begin{thm} \label{afcw-cubic-thm}
If $\mu$ as in  \eqref{afcw-cubic-def} satisfies the negative lattice condition \eqref{nlc}, then
\begin{equation} \label{bk-afcw-cubic}
\mu(A \square B) \leq \mu(A) \, \mu(B),
\end{equation}
for all increasing $A, B \subset \{-1, +1\}^n$.
\end{thm}
\noindent
This theorem will be proved in Section \ref{subs-cw3-proof}.
\end{subsubsection}
\subsubsection{A cluster-disjointness inequality for the ferromagnetic Ising model} \label{sect-fi}
In this section we state, for the ferromagnetic Ising model, a version of Theorem \ref{bkr} with a modified
form of the $\square$ operation.
First we recall that the ferromagnetic Ising measure for vertices $1, \cdots, n$, interaction
parameters $J_{i,j} \geq 0$, $1 \leq i, j \leq n$ and external fields $h_i$, $1 \leq i \leq n$,
is the distribution on $\{-1, +1\}^n$ given by

\begin{equation}  \label{fim-def}
\mu(\om) = \frac{\exp\left(\sum_{i, j} J_{i,j} \, \om_i \om_j \, +
\, \sum_i h_i \, \om_i\right)}{Z}, \,\,\, \om \in \{-1, +1\}^n.
\end{equation}
It is well-known that this measure satisfies the FKG inequality: for all increasing events $A$ and $B$,

\begin{equation}\label{fkg-ising}
\mu(A \cap B) \geq \mu(A) \mu(B).
\end{equation}
Note that (since the complement of an increasing event is decreasing) this is equivalent to saying that
for all increasing events $A$ and decreasing events $B$,

\begin{equation}\label{fkg-ising-equiv}
\mu(A \cap B) \leq \mu(A) \mu(B).
\end{equation}

To define the modified $\square$-operation we first consider the usual graph $G$ induced by the interaction
values. This is the graph with vertices $1, \cdots, n$ where two vertices $i$ and $j$ share an edge iff $J_{i,j} >0$. 
A {\em $+$ cluster} (with respect to a realization $\om$), is a connected component in the graph obtained from
$G$ by deleting all vertices $i$ with $\om_i = -1$. Similarly, $-$ clusters are defined. The term {\em spin
cluster} will be used for $+$ as well as for $-$ clusters.
If $K \subset [n]$, we use the notation $C(K)$ for the set of all vertices $i \in [n]$ for which there
is a $j \in K$ which is in the same spin cluster as $i$.
(Note that, in particular $C(K) \supset K$).
The modified $\square$
operation is defined as in \eqref{box-def}, but with the constraint $K \cap L = \emptyset$ replaced
by the stronger constraint that $C(K) \cap C(L) = \emptyset$. Note that this stronger constraint
is equivalent to saying that there is no `monochromatic' path from $K$ to $L$.

\begin{equation} \label{box-im-def}
A \boxminus B = \{\om \in \Om \, : \, \exists K, L \subset [n] \mbox{ s.t. } C(K) \cap C(L) = \emptyset, \,
[\om]_K \subset A \text{ and } [\om]_L \subset B\}.
\end{equation}

One of our main results is that, with the above modification of the $\square$-operation,
the ferromagnetic Ising model satisfies the analog of Theorem \ref{bkr}:

\begin{thm} \label{fim-thm}
The ferromagnetic Ising measure \eqref{fim-def} satisfies
\begin{equation} \label{fim-eq}
\mu(A \boxminus B) \leq \mu(A) \mu(B),
\end{equation}
for all $A, B \subset \{-1, +1\}^n$.
\end{thm}
\noindent
In Section \ref{sect-gibbs} we present a more general result of this flavour,
Theorem \ref{gibbs-thm}, for Gibbs
measures. In Section \ref{sect-fim-proof} we show that Theorem \ref{fim-thm} can be obtained easily
from Theorem \ref{gibbs-thm} (which in turn follows from our most general result, Theorem \ref{bk-gen}).

Note that if $A$ is increasing and $B$ decreasing, then $A \boxminus B = A \cap B$, so that the FKG inequality
\eqref{fkg-ising} can be considered as a special case of \eqref{fim-eq}.
Another special case is given by the following corollary. 
For $W, W' \subset [n]$ we write 
$W \stackrel{+}{\ra} W'$ for the event that there is a $+$ path (i.e. a path, with respect to the graph structure
mentioned above, of which every vertex
has value $+1$) from some vertex in $W$ to some vertex in $W'$.
We denote the complement of this event simply by $(W \stackrel{+}{\ra} W')^c$. 

\begin{cor} \label{ising-1-cor}
Let $\mu$ be the Ising distribution defined above, and let $X, Y, U, W \subset [n]$.
Then
\begin{equation} \label{is-cor-eq-1}
\mu( \exists x \in X, u\in U \mbox{ s.t. } x \stackrel{+}{\ra} Y, \, u \stackrel{+}{\ra} W, \,
(x \stackrel{+}{\ra} u)^c) \leq
\mu(X \stackrel{+}{\ra} Y) \, \mu(U \stackrel{+}{\ra} W).
\end{equation}
\end{cor}
\begin{proof} (of Corollary \ref{ising-1-cor}).
Take for $A$ the event $\{X \stackrel{+}{\ra} Y\}$ and for $B$ the event $\{U \stackrel{+}{\ra} W\}$.
Then $A \boxminus B$ is the event in the l.h.s. of \eqref{is-cor-eq-1}. Now apply Theorem \ref{fim-thm}.
\end{proof} 

\noindent
{\bf Remark:} {\em If $X$, $Y$ $U$ and $W$ have only one element, say $x$, $y$, $u$ and $w$ respectively,
the event in the l.h.s. of \eqref{is-cor-eq-1} can be written as
$$\{x \stackrel{+}{\ra} y\} \, \cap \, \cup_{K}^{(*)} ([{\bf 1}]_K \cap [{\bf -1}]_{\partial K}),$$
where the union marked by $*$ is over all connected components $K$ with $\{u, w\} \subset  K$ and
$\{x, y\} \cap K = \emptyset$, and where $[{\bf 1}]_K$  is
the event that all vertices in $K$ have value $1$, and $[{\bf -1}]_{\partial K}$ is the event that
all vertices that are not in $K$ but have a neighbour in $K$ have value $-1$.
Since this is a union of disjoint events, a more direct decoupling inequality of Borgs and Chayes \cite{BC96} can
be applied in this case to obtain \eqref{is-cor-eq-1}. However, if the sets $X$, $Y$, $U$ and $W$ have more elements,
the event in the l.h.s. of \eqref{is-cor-eq-1} can, in general, not be written as a suitable disjoint union, and
the Borgs-Chayes decoupling inequality is not applicable.
}

\smallskip
The following `four-arm event' is another example which illustrates how Theorem \ref{fim-thm} can be used.
Consider the graph with vertices $V = \{-k, \cdots,k\}^2 \setminus \{(0,0)\}$, and where
two vertices $v = (v_1, w_1)$ and $w = (w_1,w_2)$ share an edge iff 
$|v_1 - w_1| + |v_2 - w_2| = 1$. In other words, this graph is the $2 k \times 2k$ box on the
square lattice, centered at $O$, but with $O$ `cut out'. Consider the Ising distribution
on $\{-1, +1\}^V$ with external fields $h_v, v \in V$ and interaction parameters $J_{v,w} > 0$
if $v$ and $w$ share an edge and $0$ otherwise.
We use the notation $W \stackrel{+}{\ra} W'$ as in Corollary \eqref{is-cor-eq-1}, and the notation
$W \stackrel{-}{\ra} W'$ for its analog, with $+$ replaced by $-$.
Further, we will use here the notation $\partial V$ in a slightly different way as above, namely for the set of
those vertices $(v_1, v_2) \in V$ for which $|v_1| + |v_2| = k$.

\begin{cor} \label{ising-cor}
Let $\mu$ be the Ising distribution on $V$ described above. We have
\begin{eqnarray}\label{is-cor-eq}
\, &\,& \mu\big((1,0) \overset{+}{\ra} \partial V, \, (-1,0) \overset{+}{\ra} \partial V, \,
(0,1) \overset{-}{\ra} \partial V, \, (0,-1) \overset{-}{\ra} \partial V\big) \leq \\ \nonumber
\, &\,& 
\mu\big((1,0) \overset{+}{\ra} \partial V\big) \, \mu\big((-1,0) \overset{+}{\ra} \partial V\big) \,
\mu\big((0,1) \overset{-}{\ra} \partial V\big) \, \mu\big((0,-1) \overset{-}{\ra} \partial V\big).
\end{eqnarray}
\end{cor}

\noindent
\begin{proof} (of Corollary \ref{ising-cor} from Theorem \ref{fim-thm}). 
Let $A$ be the event $\{(1,0) \overset{+}{\ra} \partial V, (0,1) \overset{-}{\ra} \partial V \}$,
and $B$ the event $\{(-1,0) \overset{+}{\ra} \partial V, (0,-1) \overset{-}{\ra} \partial V \}$.
It is easy to see that the event in the l.h.s. of \eqref{is-cor-eq} is contained in $A \boxminus B$,
(with $\boxminus$ as defined in \eqref{box-im-def}).
Hence, by Theorem \ref{fim-thm},
the l.h.s. of \eqref{is-cor-eq} is at most $\mu(A) \mu(B)$, which, by applying the FKG inequality 
\eqref{fkg-ising-equiv} to $\mu(A)$
and $\mu(B)$ separately, is at most the r.h.s. of \eqref{is-cor-eq}. \end{proof}

\smallskip\noindent
{\bf Remarks:} \\
{\em
(i) At first sight one might have the impression that Corollary \ref{ising-cor} can be proved
more directly, by the earlier mentioned decoupling inequalities in \cite{BC96} or a straightforward
combination of FKG and elementary manipulations. However, we
do not see how to do that. \\
(ii) Various versions of Corollary \ref{ising-cor} can be proved in exactly the same way.
One example is the analog of this Corollary, where the `hole' in the box is bigger than just
one point, and where the event in the l.h.s. of \eqref{is-cor-eq} is replaced by
the event that there exist four points $u$, $v$, $w$ and $x$ on the boundary of the
hole (denoted by $\partial H$) with the property that travelling along
this boundary clockwise, starting in $u$, we first
encounter $v$, then $w$ and then $x$, and such that $u \overset{+}{\ra} \partial V$,
$v \overset{-}{\ra} \partial V$, $w \overset{+}{\ra} \partial V$ and $x \overset{-}{\ra} \partial V$.
(In this case the r.h.s. of \eqref{is-cor-eq} is replaced by the product of
$(\mu(\partial H \overset{+}{\ra} \partial V))^2$
and $(\mu(\partial H \overset{-}{\ra} \partial V))^2$.).
A different kind of version, which can also be proved in the same way is that where
the $-$path events are replaced by their analogs for so-called $*$ paths (i.e. where besides horizontal
and vertical steps, also diagonal steps are allowed in the path).
}

\end{subsubsection}

\begin{subsubsection}{Potts models} \label{sect-potts}
The Potts measure for vertices $1, \cdots, n$, set of `spin' values $S$, and interaction
parameters $J_{i,j}$, $1 \leq i, j \leq n$
is the distribution on $S^n$ given by

\begin{equation}  \label{afp-def}
\mu(\om) = \frac{\exp\left(\sum_{i, j} J_{i,j} \, I_{\{\om_i = \om_j\}} \right)}{Z}.
\end{equation}

If all the $J_{i,j}$'s are larger (smaller) than or equal to $0$ we say that the measure
is ferromagnetic (antiferromagnetic). Note that if $|S| = 2$, the ferromagnetic Potts measure is, in fact,
the ferromagnetic Ising model.
Remarkably, if $|S| \geq 3$ we do not have non-trivial results for the ferromagnetic Potss model, but we
do have one for the antiferromagnetic case.
First, as we did for the Ising model,
we consider the graph
with vertices $1, \cdots, n$ where two vertices $i$ and $j$ share an edge iff $J_{i,j}  \neq 0$.
A path $\pi$ in this graph is called {\em changing} (w.r.t. a configuration $\om$) if,
for each two consecutive vertices $v$ and $w$ on
$\pi$, $\om_v \neq \om_w$. (In particular, the path consiting of the vertex $v$ only, is considered as
a changing path).
The {\em cluster} of a set $K \subset [n]$, again (as in Section \ref{sect-fi}) denoted by $C(K)$, is now defined as the set of all
vertices $v$ for which there is a changing path with starting point in $K$ and endpoint $v$.

The modified operation $\boxminus$ we now use has exactly the same form as 
\eqref{box-im-def} (but now with the new meaning of $C(K)$ and $C(L)$). We get the following analog of
Theorem \ref{fim-thm}.

\begin{thm} \label{afp-thm}
Consider the Potts  measure $\mu$ (see \eqref{afp-def}) with all $J_{i,j}$'s non-positive.
With the above defined notion of clusters and $\boxminus$ operation, this measure satisfies

\begin{equation} \label{afp-eq}
\mu(A \boxminus B) \leq \mu(A) \mu(B),
\end{equation}
for all $A, B \subset S^n$.
\end{thm}

\noindent
In Section \ref{afp-proof} we will show that this theorem follows from Theorem \ref{gibbs-thm} below.

\end{subsubsection}

\begin{subsubsection}{Gibbs measures} \label{sect-gibbs}
This section sets the two previous ones in a more general context. (Since the Ising model and Potts
model are such widely used models, and the correspondence of the $\boxminus$ operation in Theorem \ref{fim-thm} and
\ref{afp-thm} to that in Theorem \ref{gibbs-thm} below is not immediate, we devoted a separate section to them).
First we give several definitions and introduce the needed notation.

If $K$ and $L$ are disjoint subsets of $[n]$, and $\al \in S^K$ and $\gamma \in S^L$,
we denote by $\al \circ \om$ the configuration on $K \cup L$ that agrees with $\al$ on $K$ and with
$\gamma$ on $L$. Formally,

$$\al \circ \gamma := \{\om \in S^{K \cup L} \, : \, \om_K = \al \text{ and } \om_L = \gamma \}.$$
A potential $\Phi$ (for the configuration space $S^{[n]}$) is a collection of functions 

$$\Phi_b \, : \, S^b \rightarrow \R, \,\,\, b \subset [n].$$
Below we will often use the term `hyperedges' for subsets of $[n]$, but sometimes we will simply
call them edges.
In most examples $\Phi_b \equiv$ some constant for most $b$'s.
For the ease of notation we will often (when there is no risk of confusion) write $\Phi(\om_b)$ instead
of $\Phi_b(\om_b)$.

A hyperedge $b$ is called {\em inefficient} (with respect to a configuration $\om \in S^b$) if $|b| \geq 2$ and,
for every $N \subset b$ and every $\si \in S^b$,

\begin{equation} \label{bad-def}
\Phi_b(\om_b) + \Phi_b(\si) \leq \Phi_b(\om_N \circ \si_{b \setminus N}) + \Phi_b(\si_N \circ \om_{b \setminus N}).
\end{equation}

\noindent
Note that if $\Phi_b$ is constant on $S^b$, then $b$ is is inefficient with respect to every $\om \in S^b$.

Let $K \subset [n]$ and $v \in [n]$. 
A {\em hyperpath} from $K$ to $v$ is a sequence
$\pi = (b_1, \cdots, b_m)$, such that $K \cap b_1 \neq \emptyset$, $b_i \cap b_{i+1} \neq \emptyset, \,\,
1 \leq i \leq m-1$ and $v \in b_m$. If (w.r.t. a certain configuration)
none of these edges $b_i, 1 \leq i\leq m$, is inefficient, we
say that $\pi$ is an efficient path (w.r.t. that configuration).
We define the cluster of $K$, denoted by $C(K)$, as the set of all vertices $v$ for which there is an efficient path
from $K$ to $v$.

\smallskip\noindent 
{\bf Remark:}
{\it
Note that, by the definition of an inefficient edge, there is always an efficient path from a vertex $v$ to
itself (namely the path which consists only of the edge $\{v\}$). Hence $C(K) \supset K$.
}

\smallskip\noindent
We define the following modification of the box-operation:

\begin{equation} \label{box-gibbs-def}
A \boxminus B = \{\om \in \Om \, : \, \exists K, L \subset [n] \mbox{ s.t. } C(K) \cap C(L) = \emptyset, \,
[\om]_K \subset A \text{ and } [\om]_L \subset B\}.
\end{equation}

\smallskip\noindent
{\bf Remark:}
{\it
We have used here the same notation as for the ferromagnetic Ising model in the previous section.
Note that, by taking for $\Phi$ the `usual potential function for the Ising model', definition 
\eqref{box-gibbs-def} becomes exactly definition \eqref{box-im-def}.
}

\medskip
The Gibbs measure for the potential $\Phi$ is the measure $\mu$ on $S^{[n]}$ given by

\begin{equation} \label{gibbs-def}
\mu(\om) = \frac{\exp\left(\sum_{b} \Phi_b\left(\om_b\right)\right)}{Z},
\end{equation}
where the sum is over all $b \subset [n]$ (or, equivalently, by adjusting $Z$, over all $b \subset [n]$ with the
property that $\Phi_b$ is non-constant on $S^b$).

The main result of this section is:

\begin{thm} \label{gibbs-thm}
The Gibbs measure \eqref{gibbs-def} satisfies
\begin{equation} \label{gibbs-eq}
\mu(A \boxminus B) \leq \mu(A) \mu(B),
\end{equation}
for all $A, B \subset S^n$.
\end{thm}

\noindent
We will prove in Section \ref{gibbs-proofs} that the above theorem follows from Theorem \ref{bk-gen}.

\smallskip\noindent

\end{subsubsection}
\end{subsection}
\end{section}

\begin{section}{General framework} \label{sect-general}

\begin{subsection}{Generalized random-cluster representations} \label{sect-gen-rcr}

As before, $S$ is a finite set (and will play the role of `single-site state space'), and
as set of `indices' (also called `vertices') we take $[n] := \{1, \cdots, n\}$. (So the state space is 
$S^n$).
The set of all subsets of a set $V$ will be denoted by $\calP(V)$. In the case $V = [n]$ we
simply write $\calP(n)$ for $\calP([n])$. Elements of $\calP(n)$ will often be called hyperedges.


We assign, to each hyperedge $b$, a random subset of $S^b$. Let $\nu$ denote the joint distribution
of this collection of subsets. So $\nu$ is a probability measure on $\prod_{b \subset [n]} \calP(S^b)$.
In the following, we will typically use the notation $\eta_b$ for a subset of $S^b$, and the notation
$\eta$ for the collection $(\eta_b, b \subset [n])$.
Further (as before) we typically denote an element of $S^n$ by $\om$. 
For $\om$ and $\eta$ as above, we say that $\om$ is compatible with $\eta$ (notation: 
$\om \sim \eta$) if $\om_b \in \eta_b$ for all $b \subset [n]$.

\begin{defn} \label{def-RCR}
Let $\mu$ be a probability distribution on $S^n$.
We say that $\mu$ has a random cluster representation (RCR) with base $\nu$ if, for all
$\om \in S^n$,

\begin{equation} \label{rcr-def2}
\mu(\om) = \frac{1}{Z} \, \sum_{\eta : \eta \sim \om} \nu(\eta).
\end{equation}
\end{defn}

A hyperedge $b$ is said to be {\em active} (w.r.t. $\eta$) if $\eta_b \neq S^b$.
Two vertices (elements of $[n]$) $v$ and $w$ are said to be {\em neighbours} (w.r.t. $\eta$)
if there is an active hyperedge $b$ such that $v \in b$ and $w \in b$. This notion gives naturally
rise to the notion of {\em clusters}: the cluster of $v$ is the set which consists of $v$ and
all $w \in [n]$ for which there exists a sequence $b_1, \cdots, b_k$ of active hyperedges
such that: $v \in b_1$, $w \in b_k$, and $b_i \cap b_{i+1} \neq \emptyset$, $1 \leq i \leq k-1$.
To emphasize that these notions depend on $\eta$, we speak of $\eta$-active, $\eta$-cluster etcetera.

\smallskip\noindent
{\bf Remarks:} \\
{\it
(i) This notion of random cluster representation is an abstraction of the usual
notion in the literature. To indicate the correspondence with the usual notion, consider as an
example the ferromagnetic Ising measure \eqref{fim-def}, with all $h_i$'s equal to $0$, and each $J_{i,j}$ either $0$
or $J$. The usual
random cluster model for this measure, introduced
by Fortuin and Kasteleyn (see e.g. \cite{FK72}; see also \cite{G06} and Chapter 10 in \cite{G10})
assigns to each edge (pair of vertices
$i, j$ for which $J_{i,j} \neq 0$) the value `open' or `closed'. The two endpoints of an open edge
`receive' the same spin value (i.e. both are $+1$ or both are $-1$). In the language of our
definition above, this is the same as assigning to the edge $(i,j)$ the `value'
$\{(-1,-1), (+1, +1)\}$ (which corresponds to being `open'), or the value
$\{-1,+1\}^b$ (which corresponds to being `closed'). The {\em base} $\nu$ in Definition \ref{def-RCR} is,
in this special case, a Bernoulli measure: each edge $b = \{i,j\}$ with $i \neq j$ and $J_{i,j} >0$ has,
independently of the other edges, $\eta_b$ equal to $\{(-1,-1), (+1, +1)\}$ with probability $p$, and
equal to $\{-1,+1\}^b$ with probability $1-p$ (where $p = 1 - \exp(- 2 J)$ as in the Fortuin-Kasteleyn
(FK) random-cluster measure).
The `extra' factor ($2$ to the number of open clusters) in the FK random-cluster measure is missing in
the equation for $\nu$. This makes our computations involving $\nu$ more elegant but has
no fundamental consequences. \\
(ii) Typically a measure $\mu$ on $S^n$ has more than one random cluster representation.
For instance, taking
$$ \nu(\eta) = \begin{cases} \mu(\om), & \mbox{ if } \eta_{[n]} = \{\om\} \mbox{ and } \eta_b = S^b \mbox{ for all } b \neq [n] \\
0, & \mbox{ otherwise } \end{cases}
$$
gives an RCR of $\mu$ which is trivial and not useful.

\noindent
(iii) Generalized random-cluster representations are not only useful for the purposes in this paper but
also interesting in themselves. Several properties will be studied in more detail in the separate paper \cite{Ga12}.

\noindent
(iv) For $q$-state Potts models on the triangular lattice a generalization of the usual random-cluster model
was obtained by Chayes and Lei (see Section 3.1 in \cite{CL06}).
}
\end{subsection}

\begin{subsection}{Foldings}\label{sect-foldings}
Let $M \subset [n]$ and $\al \in S^M$. Further, let $\beta, \gamma \in S^{M^c}$ be such that
$\beta_i \neq \gamma_i$ for all $i \in M^c$.
Finally, let $\mu$ be a distribution on $S^n$.
The following notion, but less explicitly and less generally, and not with this terminology,
plays an important role in
\cite{R00}, \cite{BJ11}.

\smallskip\noindent
\begin{defn} \label{def-folding}
The $(\al;\be,\ga)$-folded version of $\mu$ is the 
probability measure on $\prod_{i \in M^c} \{\be_i, \ga_i\}$
given by

\begin{equation} \label{fold-def}
\mabg(\om) = \frac{1}{Z} \mu(\al \circ \om) \, \mu(\al \circ \bar\om), \,\,\,\,
\om \in \prod_{i \in M^c} \{\be_i, \ga_i\},
\end{equation}
where $\bar\om$ is the unique element of $\prod_{i \in M^c} \{\be_i, \ga_i\}$ with the property
that, for each $i \in M^c$, ${\bar\om}_i \neq \om_i$. 
We call $M$ the {\em locked area} of the folding.
\end{defn}

\smallskip\noindent
{\bf Remarks:} \\
{\it
(i) Instead of $(\al;\be,\ga)$-folded version we will often say $(\al;\be,\ga)$-folding (or, simply, folding).\\
(ii) Although the definition of $\bar\om$ depends on $\be$ and $\ga$ we do not show this in our notation;
it will always be clear from the context which $\be$ and $\ga$ are meant.
We will also use this notation in more generality: if $V \subset M^c$ and
$\om \in \prod_{i \in V} \{\be_i, \ga_i\}$, $\bar\om$ is the unique element of
$\prod_{i \in V} \{\be_i, \ga_i\}$ with the property that, for each $i \in V$, ${\bar\om}_i \neq \om_i$.
And, if $F \subset \prod_{i \in V} \{\be_i, \ga_i\}$, $\bar F$ is defined as the set $\{\bar \om \, : \, \om \in F \}$. \\
(iii)
It is also obvious from
the definition that if, for some indices
$i$, we replace $\be_i$ by $\ga_i$ and vice versa, this does not change the measure $\mabg$. 
In particular, if the set $S$ has only two elements, the choice of $\be$ and $\ga$ is immaterial and therefore we
simply write $\mu^{(\al)}$ in that case.\\
(iii) We will be interested in random-cluster representations of foldings.
Note that the base $\nu$ of such a representation is
a probability measure on 
$$\prod_{b \subset M^c} \calP(\prod_{i \in b} \{\be_i, \ga_i\}).$$
}
\end{subsection}
\begin{subsection}{A general form of restricted disjoint occurrence}
For $A, B \subset S^n$ and $\om \in S^n$, we define the set of disjoint-occurrence pairs
$\calD(A,B,\om)$ as follows:

\begin{equation}\label{D-def}
\calD(A,B,\om) = \{(K,L) \, : \, K, L \subset [n], K \cap L = \emptyset, [\om]_K \subset A,
[\om]_L \subset B\}.
\end{equation}
Note that $A \square B$ can be written as

\begin{equation}\label{square-def-alt}
A \square B = \{\om \in S^n \, : \, \calD(A,B,\om) \neq \emptyset\}.
\end{equation}

Restricted forms of the disjoint-occurence operation can be obtained by replacing, in the r.h.s. of 
\eqref{square-def-alt}, $\calD(A,B,\om)$ by a subset. We already did this in Section \ref{sect-fi} 
and Section \ref{sect-potts} for Ising and Potts models, and in Section \ref{sect-gibbs} for other Gibbs measures.
More generally, let $\Psi$ be a map which
assigns to each triple $(A,B,\om)$ (where $A, B \subset S^n$ and $\om \in S^n$) a (possibly empty)
subset of $\calD(A,B,\om)$. Such a map will be called a {\em selection rule}. 
Now define the $\Psi$-restricted disjoint occurrence operation as follows:

\begin{equation}\label{psi-square-def}
A \psq B  = \{\om \in S^n \, : \, \Psi(A, B, \om) \neq \emptyset\}.
\end{equation}

\noindent
This definition depends of course on the selection rule $\Psi$. Although this dependence is not
visible in the notation $A \psq B$, it will always be clear from the context to which selection rule
it refers. In fact, in the section on the ferromagnetic Ising model and that on Gibbs measures, we already used this notation (see \eqref{box-im-def} and \eqref{box-gibbs-def}, respectively), which, as we will see in Section \ref{gibbs-proofs},
corresponds to certain particular choices of the selection rule.
Note that if for $\Psi$ we take the `obvious' selection rule $\Psi(A,B, \om) = \calD(A, B, \om)$,
then $A \psq B$ is simply $A \square B$. \\
Our general theorem for (restricted) disjoint-occurrence is the following.

\begin{thm} \label{bk-gen}
Let $A, B \subset S^n$,  $\Psi$ a selection rule and $\mu$ a probability
measure on $S^n$. If, for each $M \subset [n]$, each $\al \in S^M$ and all $\be, \ga \in S^{M^c}$ with
$\be_i \neq \ga_i$ for all $i \in M^c$, the
folding $\mu^{(\al; \be, \ga)}$ has a random-cluster representation with base $\nu^{(\al; \be, \ga)}$ 
which satisfies conditions (i) and (ii) below, then
\begin{equation} \label{bk-gen-eq}
\mu(A \psq B) \leq \mu(A) \mu(B).
\end{equation}

\smallskip\noindent
{\em Condition (i) (Symmetry):} 
For $\nu^{(\al; \be, \ga)}$-almost every $\eta$ and each $b \subset M^c$, 
$\eta_b = \bar \eta_b$. \\ 
{\em Condition (ii) (Separation):}
For all $\om \in \prod_{i \in M^c} \{\be_i, \ga_i\}$ for which $\Psi(A,B,\om \circ \al) \neq \emptyset$, there is
a pair $(K,L) \in \Psi(A,B,\al \circ \om)$
such that for
$\nu^{(\al; \be, \ga)}$-almost all $\eta \sim \om$ there is no element of $K \cap M^c$ that belongs
to the same $\eta$-cluster as an element of $L \cap M^c$. 

\end{thm}

\begin{proof}
The proof highly generalizes the overall structure of the proofs of Theorem \ref{bkr} and Theorem \ref{bkj} from
Proposition \ref{reimer} in \cite{R00} and \cite{BJ11} respectively, and combines
it with the notion of random-cluster representations. \\
It is clear that \eqref{bk-gen-eq} can be written as

\begin{equation} \label{bk-gen-eq-alt}
(\mu \times \mu)\left(\left(A \psq B\right) \times S^n\right) \leq (\mu \times \mu) (A \times B).
\end{equation}

\noindent
Let, for each $M \subset [n]$, each $\al \in S^M$, and all pairs $\be, \ga \in S^{M^c}$ with
$\be_i \neq \ga_i$, $i \in M^c$,

$$W^{(\al; \be, \ga)} := \{(\al \circ \om, \al \circ \bar\om) \, : \, \om \in \prod_{i \in M^c}\{\be_i, \ga_i\}\}.$$

It is clear that two such sets are either equal to each other or disjoint. Moreover, the union of all such sets is
$S^n \times S^n$. Hence, to prove \eqref{bk-gen-eq-alt} it is sufficient to prove that, for each of the above
mentioned sets $W^{(\al; \be, \ga)}$, 

\begin{equation} \label{bk-gen-eq-alt2}
(\mu \times \mu) \left(\left(\left(A \psq B \right) \times S^n \right) \cap W^{(\al; \be, \ga)} \right)
\leq (\mu \times \mu) \left(\left(A \times B\right) \cap W^{(\al; \be, \ga)}\right).
\end{equation}

By the definition of `foldings' (see \eqref{fold-def}), and that of random-cluster representations
(see \eqref{rcr-def2}) the l.h.s. of \eqref{bk-gen-eq-alt2} is equal to

\begin{eqnarray} \label{gen-alt-lhs}
\,&\,& \sum_{\om \in \prod_{i \in M^c} \{\be_i, \ga_i\} \, : \, \al \circ \om \in A \psq B}
\mu(\al \circ \om) \mu(\al \circ \bar\om) \\ \nonumber
&=& Z \, \sum_{\om \in \prod_{i \in M^c} \{\be_i, \ga_i\} \, : \, \al \circ \om \in A \psq B} \mu^{(\al; \be, \ga)}(\om)
\\ \nonumber
&=& \frac{Z}{Z'} \sum_{\eta} \nu^{(\al; \be, \ga)}(\eta) \,
\left\vert \left\{ \om \in \prod_{i \in M^c} \{\be_i, \ga_i\} \, : \, \om \sim \eta, \, \al \circ \om \in A \psq B \right\}
\right\vert,
\end{eqnarray}
where $Z$ corresponds to the normalizing factor in \eqref{fold-def}, and $Z'$ with the normalizing factor in
\eqref{rcr-def2}. 

\smallskip\noindent

\noindent
Similarly, the r.h.s. of \eqref{bk-gen-eq-alt2} is equal to
\begin{eqnarray} \label{gen-alt-rhs}
\,&\,& \sum_{\om \in \prod_{i \in M^c} \{\be_i, \ga_i\} \, : \, (\al \circ \om) \in A, \, (\al \circ \bar\om) \in B}
\mu(\al \circ \om) \mu(\al \circ \bar\om) \\ \nonumber
&=& Z \, \sum_{\om \in \prod_{i \in M^c} \{\be_i, \ga_i\} \, : \, (\al \circ \om) \in A, \, (\al \circ \bar\om) \in B}
\mu^{(\al; \be, \ga)}(\om) \\ \nonumber
&=& \frac{Z}{Z'} \sum_{\eta} \nu^{(\al; \be, \ga)}(\eta) \,
\left\vert \left\{ \om \in \prod_{i \in M^c} \{\be_i, \ga_i\} \, : \, \om \sim \eta, \,
\al \circ \om \in A, \, \al \circ \bar\om \in B \right\} \right\vert.
\end{eqnarray}
The theorem now follows if, for each $\eta$ with $\nu^{(\al; \be, \ga)}(\eta) >0$, the cardinality
of the set of $\omega$'s in the last line of \eqref{gen-alt-lhs} is smaller than or equal to that
in the last line of \eqref{gen-alt-rhs}. The following lemma states that this is indeed the case.

\begin{lem} \label{cardinality-lem}
Let $\eta$ be such that $\nu^{(\al; \be, \ga)}(\eta) >0$. Then

\begin{eqnarray} \label{cardinality-ineq}
\, &\,& \left\vert \left\{ \om \in \prod_{i \in M^c} \{\be_i, \ga_i\} \, :
\, \om \sim \eta, \, \al \circ \om \in A \psq B \right\} \right\vert \\ \nonumber
&\leq& 
\left\vert \left\{ \om \in \prod_{i \in M^c} \{\be_i, \ga_i\} \, : \, \om \sim \eta, \,
\al \circ \om \in A, \, \al \circ \bar\om \in B \right\} \right\vert.
\end{eqnarray}
\end{lem}

\begin{proof}
{\em (of Lemma \ref{cardinality-lem})}
Let $C_1, C_2, \cdots, C_k$ denote the $\eta$-clusters. (So, in particular,
$(C_1, \cdots, C_k)$ is a partition of $M^c$).
From Condition (i) in Theorem \ref{bk-gen} (and the definition of $\eta$-clusters) it follows that if $\om \sim \eta$,
and $\si \in \prod_{i \in M^c} \{\be_i, \ga_i\}$ satisfies 
$\sigma_{C_i} \in \{\om_{C_i}, {\bar\om}_{C_i}\}$ for all $i = 1, \cdots, k$, then also $\si \sim \eta$.
Therefore it is sufficient to show that, for each $\om \sim \eta$,

\begin{eqnarray} \label{card-ineq2}
\, &\,& \left\vert \left\{ \si \in \prod_{1 \leq i \leq k} \left\{\om_{C_i}, {\bar \om}_{C_i} \right\} \, : \,
\al \circ \si \in A \psq B \right \} \right\vert \\ \nonumber
&\leq& 
\left\vert \left\{ \si \in \prod_{1 \leq i \leq k} \left\{\om_{C_i}, {\bar \om}_{C_i} \right\} \, : \,
\al \circ \si \in A, \, \al \circ {\bar\si} \in B \right \} \right\vert.
\end{eqnarray}

\noindent
Consider the map 
$T : \prod_{1 \leq i \leq k} \{\om_{C_i}, {\bar \om}_{C_i}\} \rightarrow \{0,1\}^k$, defined by

$$(T(\si))_i = \begin{cases} 1, & \mbox{ if } \si_{C_i} = \om_{C_i} \\
0, & \mbox{ if } \si_{C_i} = {\bar \om}_{C_i} \end{cases} $$
This map is clearly a $1-1$ map, and 
\begin{eqnarray} \label{T-prop}
T(\bar\si) = \overline{T(\si)},\,\, \si \in \prod_{1 \leq i \leq k} \{\om_{C_i}, {\bar\om}_{C_i}\}.
\end{eqnarray}
Now let
\begin{eqnarray}\label{T-claim-def}
D := T\left(\left\{\si \in \prod_{1 \leq i \leq k} \left\{\om_{C_i}, {\bar \om}_{C_i}\right\} \, : \,
\al \circ \si \in A\right\}\right), \\ \nonumber
E := T\left(\left\{\si \in \prod_{1 \leq i \leq k} \left\{\om_{C_i}, {\bar \om}_{C_i}\right\} \, : \,
\al \circ \si \in B\right\}\right).
\end{eqnarray}

\smallskip\noindent
{\bf Claim}\\
{\em
(i) The image under $T$ of the set in the l.h.s. of \eqref{card-ineq2} is contained in $D \square E$. \\
(ii) The image under $T$ of the set in the r.h.s. of \eqref{card-ineq2} is equal to $D \cap {\bar E}$.
}

\smallskip\noindent
To see that the first claim holds, let $\si$ be an element of the set in the l.h.s. of \eqref{card-ineq2}.
Let
$$A(\al) := \{\de \in S^{M^c} \, : \, \al \circ \de \in A\},$$
and define $B(\al)$ analogously.
By the definition of $\Psi(A,B,\al \circ \si)$ and Condition (ii), it follows that there
are disjoint subsets $\{i_1, \cdots, i_l\}$ and $\{j_1, \cdots, j_m\}$ of $\{1, \cdots, k\}$
such that 
\begin{equation}\label{claim-proof-eq1}
\left[\si\right]_{C_{i_1} \cup \cdots C_{i_l}} \subset A(\al),
\mbox{ and }
\left[\si\right]_{C_{j_1} \cup \cdots C_{j_m}} \subset B(\al).
\end{equation}
From \eqref{claim-proof-eq1} and the definition of the map $T$ it follows that 
$\left[T(\si)\right]_{\{i_1, \cdots, i_l\}} \subset D$
and 
$\left[T(\si)\right]_{\{j_1, \cdots, j_m\}} \subset E$,
and hence that $T(\si) \in D \square E$. This shows that Claim (i) holds.
To check Claim (ii) is straightforward.

Lemma \ref{cardinality-lem} is now obtained as follows: By Claim (i) and because $T$ is a 1-1 map,
the l.h.s. of \eqref{card-ineq2} is at most $|D \square E|$, which by Proposition \ref{reimer}
is at most $|D \cap {\bar E}|$, which by Claim (ii) (and again the fact that the map $T$ is $1-1$) is
equal to the r.h.s. of \eqref{card-ineq2}.
This completes the proof of Lemma \ref{cardinality-lem}.
\end{proof}
As we saw before, Lemma \ref{cardinality-lem} completes the proof of Theorem \ref{bk-gen}
\end{proof}
\end{subsection}
\end{section}

\begin{section}{RCR for Gibbs measures, and proofs of Theorems \ref{fim-thm} - \ref{gibbs-thm}} \label{gibbs-proofs}
\begin{subsection}{Proof of Theorem \ref{gibbs-thm} from Theorem \ref{bk-gen}}
We start with the following general result for random-cluster representations of Gibbs measures,
which is also interesting in itself.
See Section \ref{sect-gen-rcr} for the definition of RCR and Section \ref{sect-gibbs} for notation and
terminology related to Gibbs measures.

First some definitions: We say that $\eta_b \subset S^b$ is {\em monotone} (w.r.t. the potential $\Phi$) if
$\om \in \eta_b$ and $\Phi_b(\om') \geq \Phi_b(\om)$ implies
$\om' \in \eta_b$. We say that the collection $\eta = (\eta_b, \, b \subset [n])$ is monotone
if each $\eta_b$, $b \subset [n]$, is monotone. Finally, we say that a probability measure $\nu$ on the
set $\prod_{b \subset [n]} \calP(S^b)$ is monotone if it is concentrated on the set of monotone $\eta$'s
(i.e. if $\nu(\eta) = 0$ whenever $\eta$ is not monotone).

\begin{lem}\label{gibbs-rcr-prop}
Let $\Phi$ be a potential for the configuration space $S^n$, as defined in Section 
\ref{sect-gibbs}, and let $\mu$ be the Gibbs measure on $S^n$ for the potential $\Phi$.
Then $\mu$ has a RCR with base $\nu$ given by
\begin{eqnarray} \label{eq-rcr-gibbs}
\,\, &\,& \nu(\eta) = \\ \nonumber
\, &\,& \frac{1}{Z} \prod_{b \subset[n]} \left(\min\{\exp(\Phi_b(\gamma)) \, : \, \ga \in \eta_b\}
- \max\{\exp(\Phi_b(\ga)) \, : \, \ga \notin \eta_b\} \right), \\ \nonumber
 \,\, &\,& \mbox{ if } \eta \mbox{ is monotone, and } 0 \mbox{ otherwise }.
\end{eqnarray}
(In \eqref{eq-rcr-gibbs} we define the maximum over an empty set to be $0$). \\
\end{lem}
\noindent
{\bf Remark:}
{\em Although
we do not need the explicit form \eqref{eq-rcr-gibbs} for the proof of Theorem \ref{gibbs-thm} (only the monotonicity of $\nu$
is needed), this form may be of interest in itself. Note from this form that $\nu$ is a product measure (where the product is
over all edges $b$ for which $\Phi_b$ is non-constant).
}

\begin{proof}(of Lemma \ref{gibbs-rcr-prop}): \\
We have to show that $\nu$ is indeed the base of an RCR for $\mu$. So Let $\om \in S^n$.
We have

\begin{eqnarray}
\,\, &\,& Z \sum_{\eta \sim \om} \nu(\eta) \\ \nonumber
\, &=& \sum_{\eta}^{*} \prod_b I(\om_b \in \eta_b)
\left(\min\{\exp(\Phi_b(\gamma)) \, : \, \ga \in \eta_b\}
- \max\{\exp(\Phi_b(\ga)) \, : \, \ga \notin \eta_b\} \right)  \\ \nonumber
\, &=& \prod_b \sum_{\beta \subset S^b \, : \, \om_b \in \beta}^{*}
\left(\min\{\exp(\Phi_b(\gamma)) \, : \, \ga \in \beta \}
- \max\{\exp(\Phi_b(\ga)) \, : \, \ga \notin \beta \} \right),
\end{eqnarray}
where the symbol $*$ in the second line indicates that the sum is over all monotone
$\eta$ (in the set  $\in \calP(\prod_{b \subset [n]} S^b)$, and the $*$
in the third line indicates that the sum is over monotone $\beta$.
It is easy to see (from the monotonicity property of $\beta$) that in this last sum everything cancels
except the term $\exp(\Phi_b(\om_b))$.
Hence
$$\sum_{\eta \sim \om} \nu(\eta) = \frac{1}{Z} \prod_b \exp(\Phi_b(\om_b)) = \frac{1}{Z'} \mu(\om),$$
which completes the proof.

\end{proof}

\medskip\noindent
Now we start with the proof of Theorem \ref{gibbs-thm}. 
\begin{proof} (of Theorem \ref{gibbs-thm}). \\
Let $\Phi$ and $\mu$ be as in the statement of the theorem.
First note that the definition of $A \boxminus B$ in Theorem \ref{gibbs-thm} is consistent with the
general definition \eqref{psi-square-def} of the $\boxminus$ operation. To see this, we just take the selection rule
$\Psi$ as follows.
$$\Psi(A, B, \om) = \{(K,L) \subset [n] \, : \, [\om]_K \subset A, \, [\om]_L \subset B, \, C(K) \cap C(L) = \emptyset\},$$
with $C(K)$ as defined in the paragraph below \eqref{bad-def}.

Let $M \subset [n]$, $\al \in S^M$, and $\be, \ga \in S^{M^c}$ with $\be_i \neq \ga_i, \, i \in M^c$.
Recall the definition of the folded measure
$\mu^{(\al; \be, \ga)}$ in \eqref{fold-def}. By \eqref{gibbs-def}, $\mu^{(\al; \be, \ga)}$ can be written as

\begin{eqnarray} \nonumber
\,\, &\,& \mu^{(\al; \be, \ga)}(\om) \\ \nonumber
\,\, &=& \frac{\exp\left[\sum_{b \subset M^c}\left(\Phi_b\left(\om_b\right)
+ \Phi_b\left({\bar\om}_b\right) + \sum_{b' \subset M} \left(\Phi_{b' \cup b}\left(\al_{b'} \circ \om_b\right)
+ \Phi_{b' \cup b}\left(\al_{b'}\circ {\bar\om}_b \right)\right)\right)\right]}{Z}, 
\end{eqnarray}
for $\om \in \prod_{i \in M^c} \{\be_i, \ga_i\}$.

\noindent
From this form it is clear that $\mu^{(\al; \be, \ga)}$ is the Gibbs measure
with the following potential $\tilde \Phi$:

\begin{equation} \label{gibb-fi-til}
{\tilde \Phi}_b(\om) = \sum_{b \subset b' \subset b \cup M} \Phi\left(\left(\al\circ\om\right)_{b'}\right)
+ \Phi\left(\left(\al\circ{\bar\om}\right)_{b'}\right), \,\,\,\, \om \in \prod_{i \in M^c} \{\be_i, \ga_i\}.
\end{equation}

\noindent
Note that
\begin{equation} \label{fi-symm}
{\tilde \Phi}_b(\om) = {\tilde \Phi}_b({\bar\om}), \,\,\,\,\om \in \prod_{i \in M^c} \{\be_i, \ga_i\}.
\end{equation}

Since $\mu^{(\al; \be, \ga)}$  is the Gibbs measure for the potential $\tilde \Phi$, we have by Lemma
\ref{gibbs-rcr-prop} that it has an RCR with base $\nu^{(\al; \be, \ga)}$ which is monotone w.r.t. $\tilde \Phi$.
To prove Theorem \ref{gibbs-thm} it is sufficient to show that this RCR satisfies Conditions (i) and (ii)
in the statement of Theorem \ref{bk-gen}. Condition (i) follows immediately from the above mentioned monotonicity
property of $\nu^{(\al; \be, \ga)}$ and from the symmetry property \eqref{fi-symm}. \\
Now we show that Condition (ii) also holds:
Let $\om \in \prod_{i \in M^c} \{\be_i, \ga_i\}$ and let $(K,L) \in \Psi(A, B, \om\circ\al)$.
Hence (by the way we chose $\Psi$) 

\begin{equation}\label{ckcl-disj}
C(K) \cap C(L) = \emptyset.
\end{equation}

Let $\eta \sim \om$ be such that $\nu^{(\al; \be, \ga)}(\eta) > 0$.
It is sufficient to show that no element of $K \cap M^c$ belongs to the same $\eta$-cluster as an element of $L \cap M^c$.
To do this, it is, by \eqref{ckcl-disj} sufficient to show that every $b \subset M^c$ which satisfies

\begin{equation} \label{b-property}
b \cap C(K) \neq \emptyset \mbox{ and } b \cap (C(K))^c \neq \emptyset,
\end{equation}
is inactive (w.r.t. $\eta$). \\
So, let $b \subset M^c$ satisfy \eqref{b-property}. Let $b' \subset [n]$ be such that $b' \supset b$.
From the definition of $C(K)$ it follows that $b'$ is inefficient (w.r.t.
$\al \circ \om$).
From the definition \eqref{bad-def} of `inefficient' it follows (by substituting in \eqref{bad-def} $b$ by $b'$,
$\om_b$ by $(\al \circ \om)_{b'}$, $\si$ by $((\al \circ {\bar \om})_{b'}$, and $N$ by
$\{ i \in b' \setminus M \, : \, \de_i = \om_i \}$) that

$$\Phi((\al\circ\om)_{b'}) + \Phi((\al \circ {\bar\om})_{b'}) \leq \Phi((\al\circ\de)_{b'}) + \Phi((\al \circ {\bar\de})_{b'}),$$
for all $\de \in \prod_{i \in b' \setminus M} \{\be_i, \ga_i\}$.
Applying this to each term in the r.h.s. of \eqref{gibb-fi-til} gives 
$${\tilde\Phi}_b(\om) \leq {\tilde\Phi}_b(\de), \,\,\, \mbox{ for all } \de \in \prod_{i \in b} \{\be_i, \ga_i\}.$$
From this (and because $\eta_b$ is monotone and $\om_b \in \eta_b$) it follows immediately that 
each $\de \in \prod_{i \in b} \{\be_i, \ga_i\}$ belongs to $\eta_b$. Hence $b$ is not active. \\
As explained above, this completes the proof of Theorem \ref{gibbs-thm}. \end{proof}

\end{subsection}
\begin{subsection}{Proof of Theorem \ref{fim-thm} from Theorem \ref{gibbs-thm}} \label{sect-fim-proof}
\begin{proof}
The Ising distribution \eqref{fim-def} is clearly a Gibbs measure with respect to the potential $\Phi$ given by:

\begin{eqnarray} \label{fim-po}
\Phi_b(\om_b) = \begin{cases} h_b \om_b, & \mbox{ if } |b| =1 \\ 
J_{i j} \om_i \om_j, & \mbox{ if } b \mbox{ is of the form } \{i, j\}, \mbox{ where } i, j \in [n], i \neq j \\ 
0, & \mbox{ otherwise}
\end{cases}
\end{eqnarray}

Let $b \subset [n]$ and $\om \in \{-1,+1\}^n$. Suppose $b$ is not inefficient w.r.t. $\om$ (in the sense of
definition \eqref{bad-def}, with
$\Phi$ as in \eqref{fim-po}). It then follows from the definitions that then $b$ is of the form
$\{i,j\}$ for some $i, j \in [n]$ with $i \neq j$ and $J_{i j} > 0$, and, moreover, that for all $x, y \in \{-1, +1\}$

$$\om_i \om_j + x y > \om_i x + \om_j y.$$
Hence $\om_i = \om_j$.

\noindent
Vice versa, if $J_{i j} > 0$ and $\om_i = \om_j$ then it follows similarly that $\{i,j\}$ is not inefficient.
This shows that the notion of clusters in Section \ref{sect-gibbs}, (with $\Phi$ given by \eqref{fim-po}) is
the same as that in Section \ref{sect-fi}. But then the meaning of $\boxminus$ in the two sections is also the same, and
Theorem \ref{fim-thm} is a special case of Theorem \ref{gibbs-thm}. 
\end{proof}
\end{subsection}
\begin{subsection}{Proof of Theorem \ref{afp-thm} from Theorem \ref{gibbs-thm}} \label{afp-proof}
\begin{proof}
The argument is quite similar to that for the ferromagnetic Ising model in the previous section.
First note that the antiferromagnetic Potts measure, \eqref{afp-def} with all $J_{i j}$'s non-positive,
is a Gibbs measure on $S^n$ with potential function $\Phi$ given by

\begin{eqnarray} \label{afp-po}
\Phi_b(\om_b) = \begin{cases} 
J_{i j} I_{\om_i = \om_j}, & \mbox{ if } b \mbox{ is of the form } \{i, j\}, \mbox{ with } i, j \in [n], i \neq j \\ 
0, & \mbox{ otherwise}
\end{cases}
\end{eqnarray}

We will use the notion of {\em inefficient} in the sense of definition \eqref{bad-def}, with
$\Phi$ as in \eqref{afp-po}.

Let $b \subset [n]$ and $\om \in S^n$. It follows immediately from the definitions that if $|b| \neq 2$ or
$b$ is of the form $\{i,j\}$, $i \neq j$, with $J_{i j} = 0$, then $b$ is inefficient.

Now suppose that $b$ is of the form $\{i, j\}$, $i \neq j$, with $J_{i j} < 0$ and that
$\om_i = \om_j$. We claim that in that case $b$ is also inefficient. If this claim holds, the above considerations imply
that if two vertices are in different clusters in the sense of Section \ref{sect-potts}, then they also are in
different clusters in the sense of Section \ref{sect-gibbs}. That, in turn, implies that $A \boxminus B$ as
defined in the former section is contained in $A \boxminus B$ as defined in the latter section, so that
Theorem \ref{afp-thm} follows indeed from Theorem \ref{gibbs-thm}. 
By the definition of `inefficient' and the form of the potential $\Phi$ in \eqref{afp-po}, to prove the claim it suffices
to show that (recall that $J_{i j} < 0$) for all $x, y \in S$ 

$$I_{\om_i = \om_j} + I_{x = y} \geq I_{\om_i = y} + I_{\om_j = x}.$$

\noindent
This last inequality can be checked straightforwardly: Since the first term in the l.h.s. is $1$, the
inequality can only fail if the r.h.s equals $2$, i.e. if both terms in the r.h.s. are $1$. However, it
follows immediately that in that case $\om_i$, $\om_j$, $x$ and $y$ are all equal, so that the l.h.s. is also equal to $2$.

This completes the proof of the claim, and thus that of Theorem \ref{afp-thm}.
\end{proof}
\end{subsection}
\end{section}
\begin{section}{Permutation invariance and proof of Theorem \ref{afcw-thm}} \label{afcw-proof}
We first state the following corollary of the general Theorem \ref{bk-gen}.
Recall from Section \ref{sect-foldings} that if $\mu$ is a distribution on $\{0,1\}^n$, $M \subset [n]$ and
$\al \in \{0,1\}^M$, the base of an RCR of the folding $\mu^{(\al)}$ is a distribution on the set of all
$\eta$ of the form $(\eta_b, \, b \subset M^c)$ with each $\eta_b$ a subset of $\{0,1\}^b$.

\begin{cor} \label{cor-perm-bk}
Let $\mu$ be a probability distribution on $\{0,1\}^n$ such that for each $M \subset [n]$ and each $\al \in \{0,1\}^M$,
the folding $\mu^{(\al)}$ has a random-cluster representation with base $\nu^{(\al)}$ such that $\nu^{(\al)}$-almost every
$\eta$ has the following two properties, (a) and (b) below.

\smallskip\noindent
Property (a): Every $\eta$-active $b$ has $|b| = 2$ and $\eta_b = \{(0,1), \, (1,0) \}$. \\
Property (b): If $b$ and $b'$ are $\eta$-active, then $b=b'$ or $b \cap b' = \emptyset$.

\smallskip\noindent
Then $\mu$ has the BK-property.
\end{cor}
\begin{proof}

Let $A$ and $B$ be increasing subsets of $\{0,1\}^n$. Let $\Psi$ be the following selection rule:

$$\Psi(\om, A ,B) = \{(K,L) \, : \, K, L \subset [n], \, K \cap L = \emptyset, \, \om \equiv 1 \mbox{ on } K \cup L\}.$$
It is easy to see that, since $A$ and $B$ are increasing,
$$A  \psq B = A \square B.$$

\noindent
Recall that if $W$ is a finite set and $\om \in \{0,1\}^W$, we use the notation $|\om|$ for $\sum_{i \in W} \om_i$.
Now let $M \subset [n]$ and $\al \in \{0,1\}^M$. We will show that the base $\nu^{(\al)}$ satisfies
Conditions (i) and (ii) of Theorem \ref{bk-gen}. First of all,
by Property (a) above it follows immediately that
Condition (i) is satisfied.
Further, let $\om \in \{0,1\}^{M^c}$, let $K, L \subset \Psi(A,B, \al \circ \om)$, and
let $\eta$ be such that $\nu^{(\al)}(\eta) >0$ and $\eta \sim \om$.
Suppose that $K \cap M^c$ and $L \cap M^c$ have an element in the same $\eta$-cluster.
From Property (b) in the statement of the Corollary it follows immediately that then there is
an $\eta$-active $b$ such that $K \cap b$ and $L \cap b$ are non-empty. Since $K \cap L = \emptyset$
and $\om \equiv 1$ on $K \cup L$ it follows that $|\om_b| \geq 2$. However,
this gives a contradiction with Property (a) in the statement of the Corollary.
Hence $K \cap M^c$ and $L \cap M^c$ have no element in the same $\eta$-cluster.
So Condition (ii) in Theorem \ref{bk-gen} is also satisfied, and it follows from that theorem that 
$\mu(A \square B) \leq \mu(A) \mu(B).$
\end{proof}

\begin{lem} \label{lem-lin-sol}
Let $\mu$ be a symmetric, permutation-invariant distribution on $\{0,1\}^n$; that is, there are
$p_0, \cdots, p_{\lfloor n/2 \rfloor} \geq 0$ with $\sum_{i=0}^{\lfloor n/2 \rfloor} p_i =1$,
such that for all $\om \in \{0,1\}^n$ with
$|\om| \leq \lfloor n/2 \rfloor$,
$\mu(\om) = \mu(\bar \om) = p_{|\om|}$. Suppose there exist $\xi_j \geq 0$, $j = 0, \cdots, \lfloor n/2 \rfloor$,
such that the following equations hold:

\begin{equation} \label{eq-lin-sol}
p_k = \sum_{j=0}^k a_{k j} \, \xi_j, \,\,\,\, k = 0, 1, \cdots, \lfloor n/2 \rfloor,
\end{equation}
with
\begin{equation} \label{akj-def}
 a_{k j} = \frac{k! (n-k)!}{j! (k-j)! (n-k-j)!}, \,\,\,\, 0 \leq j \leq k \leq \lfloor n/2 \rfloor.
\end{equation}
Then $\mu$ has a random-cluster representation with base $\nu$ such that $\nu$-almost every $\eta$ satisfies
properties (a) and (b) in Corollary \ref{cor-perm-bk}.
\end{lem}
\begin{proof}
First we prove the following  \\
{\bf Claim:} {\em Let $0 \leq j \leq k \leq n/2$. Let $\om \in \{0,1\}^n$ with $|\om| = k$.
Then the number of $\eta \sim \om$ which have exactly $j$ active edges and satisfy properties (a) and (b)
in Corollary \ref{cor-perm-bk} is equal to $a_{k, j}$. } \\
The proof of this claim is a rather straightforward application of elementary combinatorics
and we only give a brief sketch. Let $V \subset [n]$ be the set of indices $v$ for which $\om_v =1$.
`Constructing' an $\eta$ of the form in the claim corresponds to choosing a subset $W$ of size $j$ of $V$,
and `pairing' each $w \in W$ with an index $w' \in V^c$. (Each such pair $\{w,w'\}$ corresponds to an active
edge of $\eta$). Since $|V| = k$, there are ${k \choose j}$ ways to choose $W$. Next, for each choice of $W$
there are (since $|V^c| = n-k$) $\frac{(n-k)!}{(n-k-j)!}$ ways to assign to each $w \in W$ a $w' \in V^c$.
So the number of $\eta$'s of the form in the claim is
$${k \choose j} \frac{(n-k)!}{(n-k-j)!},$$
which indeed equals $a_{k, j}$, completing the proof of the claim.

\smallskip
Now we continue with the proof of Lemma \ref{lem-lin-sol}. Let $\nu$ be the probability distribution which
assigns to each $\eta$ probability

$$\nu(\eta) = \begin{cases} \frac{\xi_{|\eta|}}{Z}, & \mbox{ if } \eta \mbox{ satsifies (a) and (b) in
Coroll. \ref{cor-perm-bk}} \\
0, & \mbox{ otherwise, } \end{cases} $$
where we use the notation $|\eta|$ for the number of $\eta$-active edges, and $Z$ is a normalizing
constant. Now let $\om \in \{0,1\}^n$ with $|\om| = k \leq n/2$.
We have
$$\sum_{\eta \sim \om} \nu(\eta) = \sum_{j=0}^k \, \sum_{\eta: \eta \sim \om, \, |\eta| = j} \nu(\eta),$$
which, by the definition of $\nu$ and by the Claim in the beginning of this proof,
equals
$$(1/Z) \sum_{j=0}^k a_{k j} \, \xi_j,$$
which by \eqref{eq-lin-sol} is equal to $p_k/Z$ and hence to $\mu(\om)/Z$.
Further, if $|\om| \geq n/2$, then $|\bar \om| \leq n/2$, and, using the above, we get a similar result as follows:
$$\mu(\om) = \mu(\bar\om) = Z \sum_{\eta \sim \bar\om} \nu(\eta) = Z \sum_{\eta \sim \om} \nu(\eta),$$
where in the last equation we used that (with the above choice of $\nu$) $\nu$-almost every $\eta$ is
compatible with $\om$ if and only if it is compatible with $\bar \om$.
Hence $\nu$ is indeed the base of an RCR for $\mu$. From the definition of $\nu$ it is trivial that $\nu$-almost
every $\eta$ satisfies (a) and (b) in Corollary \ref{cor-perm-bk}. \\
This completes the proof of Lemma \ref{lem-lin-sol}.
\end{proof}

\begin{subsection}{Proof of Theorem \ref{afcw-thm}} \label{subs-afcw}
\begin{proof}
Let $M \subset [n]$ and $\al \in \{-1,+1\}^M$. Denote $|M|$ by $m$. From the definitions it follows
that the folded measure $\mu^{(\al)}$ is given by

\begin{eqnarray} \label{eq-cw-fold}
\mu^{\al}(\om) &=& \frac{1}{Z'} \, \mu(\al \circ \om) \mu(\al \circ \bar\om) \\ \nonumber
\,\,  &=& \frac{1}{Z''} \, \exp(2 J \sum_{i, j \in M^c} \om_i \om_j) \\ \nonumber
\,\, &=& \frac{1}{\tilde Z} \, \exp( - 2 J |\om| (n - m - |\om|)), \,\,\,\, \om \in \{-1,+1\}^{M^c},
\end{eqnarray}
where the first equality holds because the contributions from the external fields
for $\om$ and $\bar \om$ cancel, and the contributions from the `interaction' with
$\al$ for $\om$ and $\bar \om$ also cancel, and where we used the notation $|\om|$ for
the number of $i \in M^c$ with $\om_i = +1$. This distribution is clearly symmetric
and permutation-invariant in the sense of Lemma \ref{lem-lin-sol} (with the
`spin value' $0$ replaced by $-1$ but that is of course immaterial).
Writing $x$ for $\exp( -2 J)$ and $k$ for $|\om|$, the last expression in
\eqref{eq-cw-fold} (apart from the constant factor $1/{\tilde Z}$) becomes
$x^{k (n-m-k)}$. So, by Lemma \ref{lem-lin-sol} and Corollary \ref{cor-perm-bk}, it
is sufficient to prove the following

\begin{lem} \label{lem-cw-sol}
For each $n$ and each $x \geq 1$, the following system of linear equations
has a non-negative solution $(\xi_j, 0 \leq j \leq n/2)$:

\begin{equation} \label{px-lin-sys}
x^{k(n-k)} = \sum_j a_{k j} \xi_j, \,\,\, 0 \leq k \leq n/2,
\end{equation}
where $a_{k j}$ is given by \eqref{akj-def} if $0 \leq j \leq k \leq n/2$ and equal to
$0$ otherwise.
\end{lem}
\begin{proof} (of Lemma \ref{lem-cw-sol})
We start with some simple observations. First of all, since the matrix entries
$a_{k j}$ in the system of equations \eqref{px-lin-sys} are non-zero if and only if $j \leq k$,
the matrix has an inverse $(a^{(-1}_{j,k})_{0 \leq j, k \leq n/2}$ and the system of equations
has a unique solution, which we denote by $\xi_j(x), 0 \leq j \leq n/2$. \\
So we have to prove that if $x \geq 1$, then $\xi_j(x) \geq 0$ for all $j$.
From now on we restrict to $x \geq 1$. \\
Now observe that, since $a_{k 0} = 1$ for all $k$, it follows immediately that

\begin{equation}\label{obs-xi0}
\xi_0(x) =1,
\end{equation}
and

\begin{eqnarray}\label{obs-xij}
\xi_j(1) = \begin{cases} 1, & \mbox{ if } j = 1 \\
0, & \mbox{ if } j > 1.
\end{cases}
\end{eqnarray}


\noindent
We will study the derivatives of $\xi_j(x)$ for $j \geq 1$.
First we define, for a real function $f$,

\begin{equation}\label{eq-Dr}
(D_r f)(x) = \frac{d}{d x} \left( \frac{1}{x^{n - 2 r + 2}} f(x) \right), \,\,\, r = 1, 2, \cdots .
\end{equation}

\noindent
Next we define

\begin{equation}\label{eq-Drj}
\calD^{r,j}(x) = D_r(D_{r-1}(\cdots (D_1(x^n \xi_j(x))) \cdots)).
\end{equation}

\noindent
The key ingredient of the proof of Lemma \ref{lem-cw-sol} is the following Claim.

\begin{claim}\label{claim-Drj}
For each $j \leq n/2$ and each $r = 1, 2, \cdots j-1$,
\begin{equation}\label{eq-claim-Drj}
\calD^{r,j}(x) = x^{(n - 2 r)} \, r! \, \sum_{k=r}^j a^{(-1)}_{j k} \, a_{k r} \, x^{(k-r) (n-k-r)}.
\end{equation}

\end{claim}
Before we prove the claim we show how it is used to prove Lemma \ref{lem-cw-sol}.
Taking $r = j-1$ in the Claim gives

$${\calD}^{j-1,j}(x) = x^{n - 2 j + 2} \, (j-1)! \left( a^{(-1)}_{j, j-1} \, a_{j-1, j-1}
+ a^{(-1)}_{j \, j} \, a_{j, j-1} \, x^{n- 2 j + 1}\right).$$
This, together, with the obvious facts that $a^{(-1)}_{j, j-1} = 1/a_{j j}$ and 
$$a^{(-1)}_{j, j-1} = - \frac{a_{j, j-1}}{a_{j j} \, a_{j-1, j-1}},$$
gives

\begin{equation}\label{eq-Dj-1}
{\calD}^{j-1,j}(x) = x^{n - 2 j + 2} \, (j-1)! \, \frac{a_{j, \, j-1}}{a_{j \, j}} \, (x^{n - 2 j + 1} - 1) \geq 0.
\end{equation}

From \eqref{eq-claim-Drj} we have
\begin{equation} \label{Dr10}
{\calD}^{r, j}(1) = r! \, \sum_{k=r}^j a^{(-1)}_{j \, k} \, a_{k \, r} = 0, \,\,\, \mbox{ for } r \neq j.
\end{equation}

Using this and the definition of $\calD$ we can now go `step by step backwards', starting from \eqref{eq-Dj-1},
as follows.
From the definition we have that

$$\frac{d}{d x} \left(\frac{\calD^{j-2, \, j}(x)}{x^{n - 2 j + 4}} \right) = {\calD}^{j-1,j}(x),$$
which by \eqref{eq-Dj-1}) is $\geq 0$. Since we also have, by \eqref{Dr10}, that 
${\calD}^{j-2, \, j}(1) = 0$, it follows that
$${\calD}^{j-2, \, j}(x) \geq 0, \,\,\, \mbox{ for all } x \geq 1.$$
Repeating this argument for $j-3$, $j-4$ etcetera, we get eventually that
$${\calD}^{1, j}(x) \geq 0, \,\,\, \mbox{ for all } x \geq 1.$$
Now recall that the l.h.s. of this last expression is, by definition, $\frac{d}{d x} \xi_j(x)$. Also recall
(see \eqref{obs-xij}) that $\xi_j(1) = 0$. Hence $\xi_j(x) \geq 0$ for all $x \geq 1$, which is the
statement of the lemma.

\smallskip
So the only thing which still has to be done is to prove Claim \ref{claim-Drj}. This is done by induction.
First note that if $r=1$ then, by the definitions \eqref{eq-Dr} and \eqref{eq-Drj}, the l.h.s. of 
\eqref{eq-claim-Drj} is just $\frac{d}{d x} \xi_j(x)$, which 
we can write as
\begin{eqnarray}\label{eq-proof-claim-step1}
\frac{d}{d x} \xi_j(x) &=& \frac{d}{d x} \left(\sum_{k = 0}^j a^{(-1)}_{j k} \, x^{k (n-k)}\right) \\ \nonumber
\,\, &=& \sum_{k=1}^j a^{(-1)}_{j k} \, k (n-k) \, x^{k(n-k) - 1} \\ \nonumber
\,\, &=& \sum_{k=1}^j a^{(-1)}_{j k} \, a_{k 1} \, x^{k(n-k) - 1} \\ \nonumber
\,\, &=& x^{n-2} \sum_{k=1}^j a^{(-1)}_{j k} \, a_{k 1} \, x^{(k-1) (n-k-1)},
\end{eqnarray} 
where the third equality uses the definition \eqref{akj-def} of $a_{k j}$.
Since the last expression in \eqref{eq-proof-claim-step1} 
is equal to the r.h.s. of \eqref{eq-claim-Drj} (for $r=1$), this shows that the Claim holds for $r = 1$.
Now suppose the Claim holds for $r-1$. We show that then it also holds for $r$: By the induction
hypothesis (and the definition \eqref{eq-Drj}), the l.h.s. of \eqref{eq-claim-Drj} can be written as
\begin{eqnarray}\label{eq-proof-claim-step2}
\,\, &\,& D_r\left(x^{n- 2 r + 2} \, (r-1)! \, \sum_{k = r-1}^j a^{(-1)}_{j k} \, a_{k, \, r-1} \, 
x^{(k-r+1)(n-k-r+1)}\right)  \\ \nonumber
\, &=& (r-1)! \, \sum_{k=r}^j a^{(-1)}_{j k} \, a_{k, \, r-1} \, (k-r+1) (n-k-r+1) \, x^{(k-r+1)(n-k-r+1) - 1} \\ \nonumber
\, &=& x^{n- 2 r} \, r! \, \sum_{k =r}^j \,  a^{(-1)}_{j k} \, a_{k r} \, x^{(k-r) (n-k-r)},
\end{eqnarray}
where the first equality follows from the definition \eqref{eq-Dr} of $D_r$, and the last from simple
manipulations and the definition 
\eqref{akj-def}. Since the last expression in  \eqref{eq-proof-claim-step2} is equal to the r.h.s. of
\eqref{eq-claim-Drj}, this completes the proof of Claim \ref{claim-Drj}. \\
As we pointed out before, this also completes the proof of Lemma \ref{lem-cw-sol}.
\end{proof}
Finally, as we explained before the statement of Lemma \ref{lem-cw-sol}, this completes the proof
of Theorem \ref{afcw-thm}.
\end{proof}
\end{subsection}

\begin{subsection}{Proof of Theorem \ref{afcw-cubic-thm}} \label{subs-cw3-proof}
\begin{proof} The negative lattice condition can be expressed in terms
of the foldings as follows. For  $\omega, \omega' \in \Omega$, let
$M=\{i \in [n]\,\, : \, \omega_i = \omega'_i\}$ and let
$\alpha = \omega_M$. Then
$\mu(\omega)\mu(\omega')=\mu^{(\alpha)}(\omega_{M^c})$. Moreover,
if $\hat {\omega} \in \{-1, +1\}^{M^c}$ is such that $\hat {\omega}_i \equiv 1$,
then $\mu(\omega\vee \omega') \, \mu(\omega \wedge \omega')=\mu^{(\alpha)}(\hat {\omega})$.
The negative lattice condition is then equivalent to
\begin{equation} \label{equiv-nlac}
\mu^{(\alpha)}(\si) \geq \mu^{(\alpha)}(\hat {\omega})
\,\, \text{ for all } \si \in \Omega_{M^c}.
\end{equation}

For $\mu$ as in \eqref{afcw-cubic-def}, it is easy to see that in each
folding the interaction between an odd number of spins vanishes, and
that what is left is a permutation invariant model with interactions
expressed in terms of products of two spin values. More precisely,
the folding is of the form
\begin{equation} \label{nlac-afcw-cubic}
\mu^{(\alpha)}(\si)=
 \frac{\exp\left( J'\sum_{i, j \in M^c}  \si_i \si_j \right)}{Z}
=\frac{1}{Z}x^{k'(n'-k')}, \,\,\, \si \in \{-1,+1\}^{M^c},
\end{equation}
for a suitable $x$ (which depends on $\alpha$), where
$n'=|M^c|$ and $k'=|\si|$.

From \eqref{equiv-nlac} and \eqref{nlac-afcw-cubic} it follows that
$x \geq 1$.
Hence, by Lemma \eqref{lem-cw-sol}  there is a nonnegative solution to the system of equations \eqref{px-lin-sys}. This
implies that $\mu^{(\alpha)}$ satisfies the conditions of Lemma \eqref{lem-lin-sol},
so that Corollary \eqref{cor-perm-bk} can be applied to yield Theorem
\eqref{afcw-cubic-thm}.
\end{proof}

\end{subsection}
\end{section}

\end{document}